\documentclass[a4paper,11pt,reqno]{amsart}

\usepackage{amsmath}
\usepackage{amsfonts}
\usepackage{amssymb}
\usepackage{amsthm}

\numberwithin{equation}{section}

\newtheorem{theorem}[equation]{Theorem}
\newtheorem{lemma}[equation]{Lemma}

\newtheorem{corollary}[equation]{Corollary}

\theoremstyle{definition}
\newtheorem{definition}[equation]{Definition}
\newtheorem{example}[equation]{Example}

\theoremstyle{remark}
\newtheorem{remark}[equation]{Remark}

\def\kint_#1{\mathchoice%
          {\mathop{\kern 0.2em\vrule width 0.6em height 0.69678ex depth -0.58065ex
                  \kern -0.8em \intop}\nolimits_{\kern -0.4em#1}}%
          {\mathop{\kern 0.1em\vrule width 0.5em height 0.69678ex depth -0.60387ex
                  \kern -0.6em \intop}\nolimits_{#1}}%
          {\mathop{\kern 0.1em\vrule width 0.5em height 0.69678ex depth -0.60387ex
                  \kern -0.6em \intop}\nolimits_{#1}}%
          {\mathop{\kern 0.1em\vrule width 0.5em height 0.69678ex depth -0.60387ex
                  \kern -0.6em \intop}\nolimits_{#1}}}

\newcommand{\R}{\mathbb{R}}

\renewcommand{\mod}{\operatorname{mod}}
\newcommand{\dd}{\,d}
\newcommand{\dmu}{\,d\mu}
\newcommand{\eps}{\varepsilon}
\newcommand{\essinf}{\operatornamewithlimits{ess\, inf}}
\newcommand{\dist}{\operatorname{dist}}

\title{Strong $A_\infty$--weights are $A_\infty$--weights on metric spaces}
\author{Riikka Korte AND Outi Elina Maasalo}
\address{R.K. Department of Mathematics and Statistics, P.O. Box 68 (Gustaf H\"allstr\"omin katu 2b), FI-00014 University of Helsinki, Finland.}
\email{riikka.korte@helsinki.fi}
\address{O.E.M. Universit\"at Bern,
Mathematisches Institut,
Sidlerstrasse 5,
3012 Bern, Switzerland.}
\email{outi.elina.maasalo@tkk.fi}
\subjclass[2000]{42B35}
\keywords{metric doubling measure, metric spaces, Muckenhoupt weights, strong $A_\infty$--weight}
\begin{document}
\maketitle
\begin{abstract}
We prove that every
strong $A_\infty$--weight is a Muckenhoupt weight in Ahlfors--regular metric measure spaces that support a Poincar\'e inequality.
We also explore the relations between various definitions for $A_\infty$-weights in this setting, since some of these characterizations are needed in the proof of the main result.

\end{abstract}

\section{Introduction}
The purpose of this paper is to study strong $A_\infty$--weights and $A_\infty$--weights in Ahlfors-regular metric measure spaces. In particular, we answer to a question proposed by Costea in \cite{Cost}, and show that every strong $A_\infty$--weight is an $A_p$--weight for some $p<\infty$ also in general metric setting. The space is assumed to be Ahlfors--regular and satisfy a weak $(1,1)$--Poincar\'e inequality. We thus extend the result by Semmes \cite{Semm93} from $\R^n$ to general metric spaces. The Euclidean proof used extensively the linear structure of $\R^n$, for example convolutions and lines parallel to the coordinate axes. These tools are naturally not available in the metric setting. However, they can be replaced by more general methods. This shows, in particular, that the geometry of $\R^n$ is not crucial to the result. 

Strong $A_\infty$-weights were first introduced in $\R^n$ by David and Semmes in~\cite{DaviSemm90} and~\cite{Semm93} when trying to characterize the subclass of $A_\infty$--weights that are comparable to the Jacobian determinants of quasiconformal mappings.  Later they have studied strong $A_\infty$--weights, for example, in  \cite{Semm96}.  See also  Bonk, Heinonen and Saksman~\cite{BonkHeinSaks04} and~\cite{BonkHeinSaks08} and Heinonen and Koskela \cite{HeinKosk95} for further results concerning the quasiconformal Jacobian problem. Recently, strong $A_\infty$-weights have been studied, for example, by Costea in~\cite{Cost} and~\cite{Cost07}. In~\cite{Cost} he studies connections between strong $A_\infty$--weights and Besov and Morrey spaces, and in ~\cite{Cost07} he extends the results to the metric setting. 
Strong $A_\infty$--weights turn out to be useful in various applications, such as in studying elliptic partial differential equations, weighted Sobolev inequalities and Mumford--Shah type functionals. See, for example, \cite{BaldFran05}, \cite{Bjor02}, \cite{DiFaZamb06}, \cite{DiFaZamb05}, \cite{FranGutiWhee94}, \cite{HeinSemm97} and \cite{Laak02}.

In Euclidean spaces, there are several equivalent characterizations for $A_\infty$--weights. For example, a weight is an $A_\infty$--weight if and only if it satisfies the reverse H\"older inequality or belongs to the class $A_p$ for some finite $p$. 
Some of these relations are needed in proving that strong $A_\infty$--weights are $A_\infty$--weights. 
However, in more general spaces, all of these conditions are not necessarily equivalent, and, in particular, the class of $A_\infty$--weights can be strictly larger than the union of $A_p$--classes, see Str�omberg and  Torchinsky \cite{StroTorc89}. 
In the last section of the paper, following  \cite{StroTorc89}, we study the relations between five different conditions in general metric spaces, and, in particular, we show that strong $A_\infty$-weights satisfy all of them. 
Furthermore, we give some examples of weights that only satisfy some of the characterizations.

\section{Preliminaries}
\subsection{Assumptions on the measure}
Let $(X,d,\mu)$ be a metric measure space, where $\mu$ is Borel
regular. We assume that the space is Ahlfors $Q$--regular with
$Q\geq 1$, i.e. there exists $c_A\geq 1$ such that
\[
 \frac{1}{c_A} r^Q\leq \mu(B(x,r))\leq c_Ar^Q
\]
for all $x\in X$ and $r>0$. Notice that such a measure is always
\emph{doubling}, that is, there exists a constant $c_D\geq 1$ such
that
\[
\mu(B(x,2r))\leq c_D\mu(B(x,r))
\]
for all $x\in X$ and $r>0$. Later, $\lambda B$ denotes the ball with the same center as $B$ but $\lambda$ times its radius.
\subsection{Modulus of a curve family and Newtonian spaces}
Let $1\leq p<\infty$. For a given curve family $\Gamma$ in $X$, we define the $p$--modulus of $\Gamma$ by
\[
 \mod_p\Gamma=\inf\int_X\rho^p \dmu,
\]
where the infimum is taken over all nonnegative Borel functions $\rho:X\rightarrow[0,\infty]$ satisfying
\begin{equation}\label{eqn:modulus_condition}
 \int_\gamma\rho \,ds\geq 1
\end{equation}
for all rectifiable curves $\gamma\in\Gamma$. We remind that a
curve is rectifiable if its length is finite.

Let $u$ be a real--valued function on $X$. We recall that a nonnegative Borel
measurable function $g$ on $X$ is said to be an \emph{upper
gradient} of $u$ if for all rectifiable curves $\gamma$ joining
points $x$ and $y$ in $X$ we have
\begin{equation}\label{upper-gr}
|u(x)-u(y)|\leq \int_{\gamma}g\,ds.
\end{equation}
If the above property fails only for a set of curves that is of zero
$p$--modulus then $g$ is said to be a \emph{p--weak upper gradient} of $u$. 
Every function $u$ that has a $p$--integrable
$p$--weak upper gradient has a \emph{minimal $p$--integrable
$p$--weak upper gradient} denoted $g_u$. 

Finally, we recall that the Newtonian space $N^{1,p}(X)$
to is the collection of all $p$--integrable functions $u$ on $X$
that have a $p$--integrable $p$--weak upper gradient $g$ on $X$. For the precise definition; see, for example, \cite{Shan00}.

\subsection{Poincar\'e inequality}
We assume that $X$ satisfies a weak
$(1,1)$--Poin\--car\'e inequality, i.e. there exists constants
$c_P,\lambda >0$ such that
\[
 \kint_{B(x,r)}|u-u_{B(x,r)}|\dmu\leq c_Pr \kint_{B(x,\lambda r)}g_u\dmu
\]
for all $u\in N^{1,1}(X)$, $x\in X$ and $r>0$.

The following estimate is a consequence of the Poincar\'e
inequality; see for example Lemma 3.3 in~\cite{Bjor02} for a proof.

\begin{lemma}\label{lemma:modulus}
Let $(X,d,\mu)$ be a metric measure space, where $\mu$ is doubling
and $X$ supports a weak $(1,1)$--Poincar\'e inequality. Let
$\Gamma$ be a curve family consisting of all rectifiable curves
joining 
$B(x_0,r)$ and $X\setminus B(x_0,2r)$.
Then
\[
 \mod_1(\Gamma)\geq C \mu(B(x_0,r))/r.
\]
The constant  $C>0$ depends only on $c_D$ and $c_P$.
\end{lemma}

Next we define $A_p$-- and strong $A_\infty$--weights.

\subsection{$A_p$--weights}\label{weights}
Let $1< p<\infty$ and $1/p+1/q=1$. Let $\omega$ be a nonnegative
function on $X$. We say that $\omega$ is an \emph{$A_p$--weight},
and write $\omega\in A_p$ if there exists constant $c_\omega>0$
such that
\begin{equation*}
\bigg(\kint_B\omega \dmu\bigg)\bigg(\kint_B\omega^{1-q}\dmu\bigg)^
{p-1}\leq c_\omega
\end{equation*}
for all balls $B$ in $X$.

We say that $\omega$ is an \emph{$A_1$--weight}, and write
$\omega\in A_1$ if there exists a constant  $c_\omega>0$ such that
\[
\kint_B \omega\dmu\leq c_\omega\essinf_B\omega
\]
for all balls $B$ in $X$.

Finally,  $\omega$ is an \emph{$A_\infty$--weight} and we write
$\omega\in A_\infty$ if there exists constants $c_\omega>0$ and
$\delta>0$ such that
\[
 \frac{\int_E\omega\dmu}{\int_B \omega\dmu}\leq
 c_\omega\left(\frac{\mu(E)}{\mu(B)}\right)^\delta
\]
for all balls $B$ in $X$ and all measurable subsets $E$ of $B$.

We will discuss the relations between these definitions in Section~\ref{section:characterizations}.

\subsection{Strong $A_\infty$--weight}

Let $\nu$ be a doubling measure on $X$. We associate to $\nu$ the quasi--distance $\delta_\nu(x,y)$ on $X$ defined by
\[
 \delta_\nu(x,y)=[\nu(B(x,d(x,y)))+\nu(B(y,d(x,y)))]^{1/Q}.
\]
We say that $\nu$ is a \emph{metric doubling measure}, if
there exists a distance function $\delta:X\times X\rightarrow [0,\infty)$ and a finite constant $C>0$ such that
\begin{equation}\label{eqn:distance}
 \frac1C \delta(x,y)\leq \delta_\nu(x,y)\leq C \delta(x,y)
\end{equation}
 for all $x,y\in X$.
Moreover,  $\omega\in L_{loc}^1(X)$ is called a \emph{strong $A_\infty$--weight}, $\omega\in SA_\infty$, if it is a density of a metric doubling measure, i.e.
\[
 \dd\nu=\omega\dd\mu.
\]


Remark that the property~\eqref{eqn:distance} can be characterized
in the following way, which will be useful later.

\begin{lemma}\label{lemma:kolmioey}
 The condition~\eqref{eqn:distance} holds if and only if there is a constant $C>0$ such that for any finite sequence $x_1,x_2,\ldots,x_k$ of points in $X$, we have
\begin{equation}\label{eqn:distancelemma}
 \delta_\nu(x_1,x_k)\leq C\sum_{j=1}^{k-1}\delta_\nu(x_j,x_{j+1}).
\end{equation}
\end{lemma}
\begin{proof}
It is immediate that~\eqref{eqn:distance} implies~\eqref{eqn:distancelemma}. To prove the converse, define
\[
 \delta(x,y)=\inf\sum_{i=0}^{N-1} \delta_\nu(z_i,z_{i+1}),
\]
 where  the infimum is taken over all finite sequences $z_0=x,z_1,\ldots,z_N=y$.
Clearly $\delta(x,y)\leq\delta_\nu(x,y)$ for all $x,y\in X$, and~\eqref{eqn:distancelemma} implies that $\delta_\nu(x,y)\leq C\delta(x,y)$. It is also easy to check that $\delta(\cdot,\cdot)$ is a distance function.
\end{proof}

The following theorems give some examples of strong $A_\infty$-weights.
\begin{theorem}\label{theorem:A_1}
 Every $A_1$--weight is a strong-$A_\infty$--weight.
\end{theorem}
\begin{proof}
Notice first that the statement of Lemma~\ref{lemma:kolmioey} holds if and only if it holds with the additional restriction that $x_j\in B(x_1,2d(x_1,x_k))$ for all $j$. The proof is similar to the Euclidean case and can be found in \cite{Semm93}.  

Assume then, that $\omega\in A_1$ and $d\nu=\omega\dmu $. Let $x_1,\ldots,x_k$ be given and assume, that $x_j\in
B(x_1,2d(x_1,x_k))$ for all $j$. Write $B=B(x_1,d(x_1,x_k))$ and $B_j=B(x_j,d(x_j,x_{j+1}))$ for $j=1,\ldots,k-1$.

Notice also, that since $\nu$ is doubling, it readily follows from the definition of $\delta_\nu$  that for all $x,y\in X$, we have
\begin{equation}
\label{eqn:comparison}
\nu(B(x,d(x,y)))^{1/Q}\leq \delta_\nu(x,y) \leq C \nu(B(x,d(x,y)))^{1/Q},
\end{equation}
where $C$ depends only on the doubling constant.  Then \eqref{eqn:comparison} implies that
\begin{equation*}
 \begin{split}
   \delta_\nu(x_j,x_{j+1})&\geq  \bigg(\int_{B_j}\omega\dmu \bigg)^{1/Q} 
\geq \big(\essinf_{B_j}\omega\big)^{1/Q} \mu(B_j)^{1/Q} \\
&\geq \big(\essinf_{6B}\omega\big)^{1/Q} \mu(B_j)^{1/Q}.
 \end{split}
\end{equation*}
Summing the above inequality over $j=1,\ldots,k-1$ we get
\begin{equation*}
 \begin{split}
   \sum_{j=1}^{k-1}\delta_\nu(&x_j,x_{j+1}) \geq  \big(\essinf_{6B}\omega\big)^{1/Q}\sum_{j=1}^{k-1}\mu(B_j)^{1/Q}  \\
& \geq C \big(\essinf_{6B}\omega\big)^{1/Q} \sum_{j=1}^{k-1}d(x_j,x_{j+1})  \geq C\big(\essinf_{6B}\omega\big)^{1/Q} d(x_1,x_k)\\
&\geq C\big(\essinf_{6B}\omega\big)^{1/Q} \mu(B)^{1/Q},
\end{split}
\end{equation*}
where we used the triangle inequality and the $Q$--regularity of $\mu$. Moreover, since $\omega$ is an $A_1$--weight, and $\mu$ is doubling, we get
\begin{equation*}
 \begin{split}
\big(\essinf_{6B}\omega&\big)^{1/Q} \mu(B)^{1/Q}  \geq C\bigg(\kint_{6B} \omega\dmu\bigg)^{1/Q} \mu(B)^{1/Q}\\
& \geq C \bigg(\int_{B} \omega\dmu\bigg)^{1/Q} 
\geq C \delta_{\nu}(x_1,x_k).
\end{split}
\end{equation*}
The proof follows now from Lemma \ref{lemma:kolmioey}. 
\end{proof}

\begin{theorem}\label{theorem:jacobian}
Let $(X,d_X,\mu_X)$ and $(Y,d_Y,\mu_Y)$ be locally compact $Q$--regular metric measure spaces such that $X$ supports a weak $(1,p)$--Poincar\'e inequality for some $p<Q$ and let $f:X\rightarrow Y$ be a quasisymmetric mapping. Then the Jacobian of $f$ is a strong $A_\infty$-weight.
\end{theorem}
\begin{proof}
We write 
\[
L(x,r)= \sup_{d_X(x,y)\leq r}{d_Y(f(x),f(y))},\quad l(x,r)= \inf_{d_X(x,y)\geq r}{d_Y(f(x),f(y))},
\] 
and recall that $f$ is quasisymmetric, if it is homeomorphism with a positive and finite constant $K$ such that $L(x,r)\leq K\, l(x,r)$ for all $x\in X$ and $r>0$. Moreover, remember that the generalized Jacobian is defined as
\[
J_f(x)=\lim_{r\to 0}\frac{\mu_Y(f(B(x,r)))}{\mu_X(B(x,r))}.
\]
By the Lebesgue-Radon-Nikodym theorem, the limit exists almost everywhere, and
\begin{equation}
\label{eqn:jacob}
 \int_EJ_f\,d\mu_X = \mu_Y(f(E))
\end{equation}
for all measurable $E\subset X$, since the measures are absolutely continuous. 

Consider $d\nu=J_f\,d\mu_X$. Let $x,y\in X$ and write $d(x,y)=r$. Then
\begin{equation*}
 \begin{split}
  d_Y(f(x),f(y))^Q&\leq L(x,r)^Q\leq K^Q l(x,r)^Q \\ 
&\leq CK^Q\mu_Y(B(f(x),l(x,r))) \leq CK^Q\mu_Y(f(B(x,r)))\\
&= CK^Q  \int_{B(x,r)}J_f\,d\mu_X  \leq CK^Q \delta_\nu(x,y)^Q.
 \end{split}
\end{equation*}
Here we used the $Q$--regularity and quasisymmetricity. On the other hand,
\begin{equation*}
 \begin{split}
\delta_\nu(x,y)^Q &= \int_{B(x,r)}J_f\,d\mu_X + \int_{B(y,r)}J_f\,d\mu_X  \\
&= \mu_Y(f(B(x,r))) + \mu_Y(f(B(y,r))) \\
&\leq \mu_Y(B(f(x),L(x,r))) + \mu_Y(B(f(y),L(y,r))) \\
&\leq  C\big(L(x,r)^Q + L(y,r)^Q\big) \leq CK^Q \big( l(x,r)^Q+l(y,r)^Q \big) \\
&\leq 2CK^Q d_Y(f(x),f(y))^Q.
   \end{split}
\end{equation*}
Since $d_Y(f(\cdot),f(\cdot))$ is a distance on $X$, the claim follows.
\end{proof}


\section{Main result}

In this  section, we prove that strong $A_\infty$--weights are $A_\infty$-weights in Ahlfors regular metric spaces.
First, we recall the Gehring lemma. A proof can be found, for
example, in \cite{Maas08},~\cite{Zato05} and~\cite{BjorBjor}.

\begin{theorem}
\label{lemma:gehring} Let $1<p<\infty$ and  assume that $f\in
L^1_{loc}(X)$ is nonnegative and defines a doubling measure. If
there exists a constant $c$ such that $f$ satisfies the reverse
H\"older inequality
\begin{equation}
\label{rhi2} \left(\kint_{B}f^p\dmu \right)^{1/p} \leq c
\kint_{B}f\dmu
\end{equation}
for all balls $B$ of $X$, then there exists positive constants  $\eps$ and $c_\eps$ such that
\begin{equation}
\label{improved-rhi} \left(\kint_{B}f^{p+\eps}\dmu
\right)^{1/(p+\eps)} \leq c_\eps \kint_{B}f\dmu
\end{equation}
for all balls $B$ of $X$. The constant $c_\eps$ as well as $\eps$
depend only on the doubling constant, $p$, and on the constant in
\eqref{rhi2}.
\end{theorem}


Now we are ready to state our main result.
\begin{theorem}
Every strong $A_\infty$--weight is an $A_\infty$--weight.
\end{theorem}
\begin{proof}



First, we construct a set of measures $\{\nu_t\}_{t>0}$ that
approximate $\nu$. Then we show that the measures
$\{\nu_t\}_{t>0}$ satisfy a reverse H\"older inequality with
uniform constants.
%
Fix $t>0$. Let $\{B_i^t=B(x_i,t)\}_{i=1}^\infty$ be a collection
of balls such that
\[
 X=\bigcup_{i=1}^\infty B_i^t
\]
and
\[
B(x_i,t/5)\cap B(x_j,t/5)=\emptyset\textrm{ for all }i\neq j.
\]
Note that the doubling property of $\mu$ implies that
\begin{equation}
\label{overlap}
 \sum_{i=1}^\infty\chi_{2B_i^t}<C.
\end{equation}
To construct a partition of unity, we define cut--off functions
\[
 \widetilde\phi_i^t(x)=
\begin{cases}
 1,& x\in B_i^t,\\
1-\dist(x,B_i^t)/t,&x\in 2B_i^t\setminus B_i^t, \\
0,&  x\in X\setminus 2B_i^t,
\end{cases}
\]
and we set
\[
 \phi_i^t=\frac{\widetilde\phi_i^t}{\sum_{i=1}^\infty \widetilde\phi^t_i}.
\]
Let
\[
 a_i^t=\frac{\int_X\phi_i^t \,d\nu}{\int_X\phi_i^t \dmu}.
\]
Since both $\mu$ and $\nu$ are doubling, and
\[
 \frac1C\chi_{B_i^t}\leq \phi_i^t\leq \chi_{2B_i^t},
\]
we have
\begin{equation}\label{eqn:a_i}
 \frac1C \frac{\nu(B_i^t)}{\mu(B_i^t)} \leq \frac1C \frac{\nu(B_i^t)}{\mu(2B_i^t)}\leq a_i^t\leq C\frac{\nu(2B_i^t)}{\mu(B_i^t)}\leq C\frac{\nu(B_i^t)}{\mu(B_i^t)}.
\end{equation}
Finally, we define the measures $\nu_t$, $t>0$ as
\[
\nu_t(A) =\sum_{i=1}^\infty a_i^t \int_A\phi_i^t \dmu.
\]
Thus
\[
 d\nu_t=\omega_t\dmu=\left(\sum_{i=1}^\infty a_i^t\phi_i^t\right)d\mu.
\]
The doubling property of $\mu$ and $\nu$ together with~\eqref{eqn:a_i} imply that for every $x,y\in X$ such that $d(x,y)\leq 2t$, we have
\begin{equation}\label{eqn:w_t_vakio}
 \frac1C\omega_t(x)\leq \omega_t(y)\leq C\omega_t(x).
\end{equation}
More precisely, we have
\begin{equation}\label{eqn:w_t:n koko}
 \frac1C\frac{\nu(B(x,t))}{\mu(B(x,t))}\leq \omega_t(y)\leq C\frac{\nu(B(x,t))}{\mu(B(x,t))},
\end{equation}
where $C$ depends only on the doubling constants of $\mu$ and $\nu$.

Now fix a ball $B(x_0,r_0)$ in $X$. Let $\Gamma$ be the set of all
rectifiable curves $\gamma:[0,L]\rightarrow X$ parametrized by arc
length such that $\gamma(0)\in B(x_0,r_0/2)$ and
$\gamma(L)\in\partial B(x_0,r_0)$. Fix $\gamma\in \Gamma$.

If $t\geq r_0$, then by~\eqref{eqn:w_t_vakio},
\[
  \frac1C\omega_t(x_0)\leq \omega_t(\gamma(s))\leq C\omega_t(x_0)
\]
for every $s\in[0,L]$ and thus we have
\[
\begin{split}
 \delta_\nu(\gamma(0),\gamma(L))\leq&C\nu(B(x_0,r_0))^{1/Q}
=Cr_0\frac{\nu(B(x_0,r_0))^{1/Q}}{\mu(B(x_0,r_0))^{1/Q}}\\
\leq& CL\omega_t(x_0)^{1/Q}\leq
C\int_0^L\omega_t(\gamma(z))^{1/Q}\,dz.
\end{split}
\]
Here we also used the doubling property of $\nu$, the fact that
\[
r_0/2\leq d(\gamma(0),\gamma(L))\leq 2r_0,
\]
and Ahlfors $Q$--regularity of $\mu$.

Now condiser the case $t< r_0$. Let $k$ be the integer part of $L/t$.
Using the previous estimate, Lemma~\ref{lemma:kolmioey} and the doubling property of $\nu$, we obtain
\[
\begin{split}
&\int_0^L\omega_t(\gamma(s))^{1/Q}\,ds
=\sum_{j=1}^k\int_{(j-1)t}^{jt}\omega_t(\gamma(s))^{1/Q}\,ds+ \int_{kt}^L\omega_t(\gamma(s))^{1/Q}\,ds\\
&\geq 1/C\sum_{j=1}^k \nu(B(\gamma(jt),t))^{1/Q}
\geq 1/C\sum_{j=1}^k \delta_\nu(\gamma((j-1)t),\gamma(jt)) \\
&\geq 1/C\delta_\nu(\gamma(0),\gamma(L))
\end{split}
\]
Thus for every $\gamma\in \Gamma$ we have
\[
 \nu(B(x_0,r_0))^{1/Q}\leq C\int_\gamma w_t^{1/Q}\,ds.
\]
If we define
\[
 \rho =
 \frac{C}{\nu(B(x_0,r_0))^{1/Q}}\omega_t^{1/Q}\chi_{B(x_0,r_0)},
\]
then $\rho$ satisfies~\eqref{eqn:modulus_condition} for every
$\gamma\in \Gamma$ and consequently
\[
 \mod_1(\Gamma)\leq \int_X \rho \dmu=\frac{C}{\nu(B(x_0,r_0))^{1/Q}}
 \int_{B(x_0,r_0)}\omega_t^{1/Q}\dmu.
\]
This combined with Lemma~\ref{lemma:modulus} gives
\begin{equation}\label{eqn:w_t_rh}
\begin{split}
&\left(\kint_{B(x_0,r_0)}\omega_t\dmu\right)^{1/Q}
 \leq C\left(\kint_{B(x_0,r_0+4t)}\omega\dmu\right)^{1/Q} \\
&\qquad\leq C\left(\kint_{B(x_0,r_0)}\omega\dmu\right)^{1/Q} \leq
C\kint_{B(x_0,r_0)}\omega_t^{1/Q}\dmu,
\end{split}
\end{equation}
where $C$ is independent of $t$, $x_0$ and $r_0$. The first two inequalities above follow from the definition of $\nu_t$ and the doubling property of $\nu$.

If we set $f=\omega_t^{1/Q}$, the Gehring lemma
\ref{lemma:gehring} now implies that there exists $\eps>0$ such
that
\begin{equation*}
\left(\kint_{B(x_0,r_0)}\omega_t^{1+\eps/Q}\dmu\right)^{\frac{1}{1+\eps/Q}}
\leq C\kint_{B(x_0,r_0)}\omega_t\dmu,
\end{equation*}
with $C$ independent of $t$, $x_0$ and $r_0$. By
Lemma~\ref{equivalence} in the next section this implies that $\omega_t$ is an $A_\infty$--weight, and there exist $p>1$ and $C>0$, independent
on $t$, such that
\begin{equation}
\label{ainfinity:v_t}
\frac{\nu_t(E)}{\nu_t(B)}\leq
C\bigg(\frac{\mu(E)}{\mu(B)}\bigg)^{1/p}
\end{equation}
 for all balls $B$ and measurable subsets $E\subset B$.

Next, we show that
\[
\nu_t\rightarrow \nu
\]
weakly in the sense of measures as $t\rightarrow 0$. In order to do that, fix an open set $U\subset X$.
Denote
\[
 U_\eps=\{x\in U\,:\, d(x,X\setminus U)>\eps\}
\]
and
\[
 I^t=\{i \,:\, 2B_i^t\subset U\} =\{i\,:\, \phi^t_i=0\textrm{ in } X\setminus U\}.
\]
By the definition of $\nu_t$, we have
\[
 \begin{split}
  \nu_t(U)=&\sum_{i=1}^\infty \frac{\int_X\phi_i^t \,d\nu}{\int_X\phi_i^t \dmu}\int_U
  \phi_i^t\dmu\geq  \sum_{i\in I^t}\frac{\int_X\phi_i^t \,d\nu}{\int_X\phi_i^t \dmu}\int_U \phi_i^t\dmu
\\=&\int_X\sum_{i\in I^t}\phi_i^t\,d\nu \geq \nu(U_{4t}).
 \end{split}
\]
Thus
\[
 \liminf_{t\rightarrow0}\nu_t(U)\geq\liminf_{t\rightarrow
 0}\nu(U_{4t})=\nu(U).
\]
Since this holds for all open sets $U\subset X$, the claim
follows.



Next we show that \eqref{ainfinity:v_t} holds true for $\nu$, and
thus $\omega$ is an $A_\infty$ weight.
 To this end, fix a ball $B$, a measurable set $E\subset B$ and an
open set $V$ such that $E\subset V\subset B$. 
Note that the weak convergence of $\nu_t$ implies
 that
\[
 \limsup_{t\rightarrow 0}\nu_t(S)\leq \nu(S)
\]
for all closed sets $S\subset X$, and by the doubling property of $\nu$ we have
\[
 \nu(\overline B)\leq \nu(2B)\leq c_D\nu(B)
\]
for all balls $B\subset X$.
Consequently, 
\begin{equation*}
\begin{split} \frac{\nu(E)}{\nu(B)}\leq C\frac{\nu(V)}{\nu(2B)} \leq
C\frac{\nu(V)}{\nu(\overline{B})} &\leq
C\frac{\liminf_{t\rightarrow0}\nu_t(V)}{\limsup_{t\rightarrow0}\nu_t(\overline{B})}
\\
&\leq C\liminf_{t\rightarrow0}\frac{\nu_t(V)}{\nu_t(B)} \leq
C\bigg(\frac{\mu(V)}{\mu(B)}\bigg)^{1/p} \end{split}
\end{equation*}
Since $\mu$ is Borel regular, taking infimum over all such $V$
finishes the proof.
\end{proof}

\section{Characterizations for $A_\infty$-weights}
\label{section:characterizations}
There are several equivalent characterizations for $A_\infty$-weights in the Euclidean setting. However, all of them are not necessarily equivalent in general metric spaces. In this section, we study the relationship between these conditions in metric spaces that are only assumed to statisfy the doubling condition. Most of these results can be found in~\cite{StroTorc89}, but for completeness, we have included the proofs here.

Recall the definitions of $A_p$--weights from Section~\ref{weights}. It follows immediately from the definitions that
%
%
for every $1<p<q<\infty$ we have
\[
A_1\subset A_p\subset A_q\subset A_\infty
\]
Moreover, in the Euclidean case,
\[
 A_\infty =\bigcup_{p<\infty} A_p.
\]
Also in the metric setting, an $A_p$--weight is always an
$A_\infty$--weight, but there exist metric spaces, where the class of $A_\infty$-weights is strictly larger than the union; see \cite{StroTorc89} and Example \ref{example1}.

Next, we state five conditions that are equivalent 
in the Euclidean setting.
For more definitions and the
Euclidean case; see \cite{GarcdeFr85}.
We consider a slightly more general situation first.
Let $\nu$ be
an arbitrary measure on $X$.  We say that $\nu$ is a weighted measure with respect to $\mu$ is there exists $\omega\in L^1_{loc}(X)$ such that for every $\mu$--measurable set $A\subset X$ we have
\[
 \nu(A)=\int_A\omega\dd\mu.
\]

We define the conditions:

\begin{enumerate}
\item\label{comparability} There are $0<\eps,\delta<1$ such that
for each ball $B$ and each measurable set $E\subseteq B$, we have
$\nu(E)\leq(1-\delta)\nu(B)$ whenever $\mu(E)\leq\eps\mu(B)$.
\item\label{ainfinity} There are constants $c>0$ and $p\geq 1$
such that
\[
\frac{\nu(E)}{\nu(B)}\leq c\bigg(\frac{\mu(E)}{\mu(B)}\bigg)^{1/p}
\]
for each ball $B$ and each measurable set $E\subseteq B$.
\item\label{rhi} $\nu$ is a weighted measure with respect to
$\mu$, and there exist positive constants $\eps\,,C$ such that the
weight $\omega$ satisfies the reverse H\"older inequality
\[
\left(\kint_{B}\omega^{1+\varepsilon}\dmu\right)^{1/(1+\varepsilon)}\leq
C\kint_B \omega\dmu
\]
for all balls $B$.
\item\label{reversed} $\nu$ is a weighted
measure with respect to $\mu$, and there exist constants $c>0$,
$p\geq 1$ such that
 \[
\frac{\nu(E)}{\nu(B)}\geq c\bigg(\frac{\mu(E)}{\mu(B)}\bigg)^{p}
\]
for each ball $B$ and each measurable set $E\subseteq B$.
\item\label{ap} $\nu$ is a weighted measure with respect to $\mu$,
and the weight $\omega$ is in $A_p$ for some $p>1$.
\end{enumerate}

\medskip
If only the measure $\mu$ is assumed to be doubling, we obtain the following relations between the conditions above in the metric setting:
\[
\eqref{comparability}
\Leftarrow \eqref{ainfinity}\Leftrightarrow\eqref{rhi}\Leftarrow\eqref{reversed}\Leftrightarrow\eqref{ap}.
\]

First we make some immediate remarks.  The condition
\eqref{comparability} follows easily from \eqref{ainfinity}, and,
since \eqref{comparability} is symmetric with respect to $\nu$ and
$\mu$, it follows also from \eqref{reversed}. Condition
\eqref{reversed} follows from \eqref{ap} by applying H\"older
inequality on
\[
\mu(E)=\int_{B}\chi_{E}\omega^{1/p}\omega^{-1/p}\dmu.
\]
In addition \eqref{reversed} implies that $\nu$ is doubling.

The following example shows that in general metric spaces, the conditions $\eqref{rhi}$ and $\eqref{reversed}$ are not necessarily equivalent:
\begin{example}
\label{example1}
Let $X=\{x=(x_1,x_2)\in \R^2\,:\, x_1\in\{1,2\},x_2\in[0,1]\}$.
We endow $X$ with the metric
\[
 d(x,y)=\begin{cases}
         |x_2-y_2|, & x_1=y_1,\\
	 2,	& x_1\neq y_1.
        \end{cases}
\]
and the onedimensional Lebesgue measure. Now let $\omega((1,x_2))=0$ ja $\omega((2,x_2))=1$. This weight clearly satisfies condition \eqref{ainfinity}, but it cannot satisfy condition \eqref{reversed} since it is not doubling. 
\end{example}

%

\begin{lemma}
\label{equivalence} \eqref{ainfinity} $\Leftrightarrow$
\eqref{rhi}
\end{lemma}
\begin{proof}
We give a sketch of the proof. If we assume \eqref{ainfinity}, the
absolute continuity  part in \eqref{rhi} is clear. We fix a
ball $B$ and write \mbox{$E_\lambda=\{x\in B \colon
\omega(x)>\lambda\}$}. Then by \eqref{ainfinity}, we have
\begin{equation*}
\mu(E_\lambda) \leq\frac{1}{\lambda}\nu(E_\lambda)
\leq\frac{c}{\lambda}\nu(B)\bigg(\frac{\mu(E_\lambda)}{\mu(B)}\bigg)^{1/p},
\end{equation*}
and, hence
\[
\mu(E_\lambda)\leq \min \{ \mu(B),
 c\big(\nu(B)^q/(\lambda\mu(B)^{1/p})^q\big)\}.
\]
Now \eqref{rhi} follows from
 \begin{equation*}
\begin{split}
\int_{B}\omega^{1+\eps}\dmu =(1+\eps)\int_{0}^{\infty}\lambda^{\eps}\mu(E_{\lambda})\,d\lambda  
\end{split}
\end{equation*}
with $0<\eps <q-1$.

On the other hand,  \eqref{ainfinity} follows from \eqref{rhi} by
applying first the H\"older and then the reverse H\"older
inequality to $\nu(E)=\int_{B}\chi_E\omega\,d\mu$.

Notice, that the doubling property of $\mu$ is not needed here.
\end{proof}

The proof of the following lemma is similar to the Euclidean case; see \cite{GarcdeFr85}.

\begin{lemma}\label{lemma3.3}
\label{together} \eqref{ainfinity} $\&$ \eqref{reversed}
$\Rightarrow$ \eqref{ap}
\end{lemma}


The proof of the following theorem is based on ideas in \cite{StroTorc89}. However, the  proof we present here is organized in a different way and contains more details.

\begin{theorem}
If $\nu$ is doubling, then\label{theorem13} \eqref{comparability}
$\Rightarrow$ \eqref{ainfinity} and \eqref{comparability}
$\Rightarrow$ \eqref{reversed}.
\end{theorem}

In order to prove Theorem~\ref{theorem13} we introduce the
notion of telescoping sequences of sets.

\begin{definition}
Let $s>0$. We say that $\{\mathcal F_k\}_{k=1}^{k_0}$ is a
$s$--\emph{telescoping sequence of collections of balls},
$\mathcal F_k=\{B_{i,k}\}_{i=1}^\infty$, provided that
\begin{itemize}
 \item $B_{i,k}\cap B_{j,k}$, for each $i\neq j$ and $k$.
\item For each $B\in\mathcal F_k$, $k=1,2,\ldots,k_0-1$, there
exists $\widetilde B\in \mathcal F_{k+1}$ such that
\[
 sB\subset s\widetilde B.
\]

\end{itemize}

\end{definition}

The following lemma is a standard covering argument, see for example Theorem 1.2 in~\cite{Hein01}. We have formulated it here to emphasize the fact that the cover can be chosen in such a way that every ball in $\mathcal F$ is included in $5B$ for some $B\in\mathcal G$.
\begin{lemma}\label{lemma:covering}
Every family $\mathcal F$ of balls of uniformly bounded diameter
in a metric space $X$ contains a disjointed subfamily $\mathcal G$
such that
\[
 \bigcup_{B\in\mathcal F}B\subset \bigcup_{B\in\mathcal G}5B.
\]
In fact, every ball $B$ from $\mathcal F$ meets a ball from
$\mathcal G$ with radius at least half that of $B$.
\end{lemma}

Now we are ready to prove Theorem~\ref{theorem13}.

\begin{proof}[Proof of Theorem~\ref{theorem13}]
We show that \eqref{comparability} implies \eqref{ainfinity}. 
Note
that in this proof, $c_D$ denotes a constant that only depends on
the doubling constants of $\mu$ and $\nu$, but it is not
necessarily exactly the doubling constant.

Fix a ball
$B_0=B(x_0,R)$ and a measurable set $E\subset B_0$.
To prove the assertion, we will construct a $5$-telescoping
sequence of collections of balls $\mathcal F_k$, $k=1,2,\ldots,
k_0,$ where $k_0$ is an integer such that
\begin{equation}\label{eqn:k_0}
(\eps/c_D^2)^{k_0+2} < \mu(E)/\mu(B_0)\leq (\eps/c_D^2)^{k_0+1},
\end{equation}
and the following properties hold: If
\begin{equation}\label{E_k}
 E_k:=\bigcup_{B\in\mathcal F_k} B
\quad\textrm{and}\quad
 \widetilde E_k:=\bigcup_{B\in\mathcal F_k} 5B,
\end{equation}
then we have $E\subset\widetilde E_1$, $\widetilde
E_{k_{0}}\subset 5 B_0$, and
\begin{equation}\label{eqn:nu-estimaatti}
 \nu(\widetilde E_{k-1}\cap B)\leq (1-\delta)\nu(B)
\end{equation}
for all $B\in\mathcal F_k$. Here $\delta$ is as in
\eqref{comparability}. Note that~\eqref{eqn:k_0} implies that
\begin{equation}\label{k_0}
 k_0\geq \frac{\log(\mu(E)/\mu(B_0))}{\log(\varepsilon/c_D^2)}-2.
\end{equation}

We may assume that $\mu(E)/\mu(B_0)$ is small enough so that $k_0$
is positive, because if $\mu(E)/\mu(B_0)$ is bigger than any fixed
constant, choosing $c$ big enough makes the right--hand side
of~\eqref{ainfinity} bigger than one.

Once such a telescoping sequence of collections of balls has been
constructed, the conclusion follows since
\begin{equation}\label{eqn:Etilde_estimaatti}
\begin{split}
\nu(\widetilde E_{k-1}) & = \nu(\widetilde E_{k-1}\cap E_k)+
\nu(\widetilde E_{k-1}\setminus E_k) \\
& \leq \sum_{B\in \mathcal{F}_k}\nu(\widetilde E_{k-1}\cap B)+
\nu(\widetilde E_{k}\setminus E_k)
 \leq (1-\delta/c_D)\nu(\widetilde E_k).
\end{split}
\end{equation}
Here we used the fact that $E_k$ is a union of disjoint balls
satisfying~\eqref{eqn:nu-estimaatti},
$\widetilde{E}_{k-1}\subset\widetilde E_k$, and that
\[
\nu(\widetilde{E}_k) \leq c_D \nu(E_k).
\]
The last estimate above follows from the doubling property of
$\nu$.
Iterating~\eqref{eqn:Etilde_estimaatti}, we obtain
\[
 \nu(E)/\nu(B_0)\leq c_D\nu(E)/\nu(5B_0)\leq c_D\nu(\widetilde E_1)/\nu(\widetilde E_{k_0})\leq
 c_D(1-\delta/c_D)^{k_0-1}.
\]
Then \eqref{ainfinity} follows by~\eqref{k_0} with constants $c$
and $p$ depending only on $\delta$, $\varepsilon$ and the doubling
constants of $\mu$ and $\nu$. Now it remains to construct the
$\mathcal F_k$'s.

We start with $\mathcal F_1$. Let $x\in E$ be a Lebesgue point of
$X$. Then we have
\begin{equation*}
\lim_{r\rightarrow 0} \frac{\mu(E\cap B(x,r))}{\mu(B(x,r))}=1,
\end{equation*}
and hence there exists $r_\eps>0$ such that
\begin{equation}
\label{lowerbound}
 \frac{\mu(E\cap B(x,r_\eps))}{\mu(B(x,r_\eps))}>\varepsilon.
\end{equation}
On the other hand, for all $r>0$,
\begin{equation}
\label{E-finite}
 \frac{\mu(E\cap B(x,r))}{\mu(B(x,r))} \leq \frac{\mu(E)}{\mu(B(x,r))},
\end{equation}
where the right--hand side tends to zero as $r$ tends to infinity
since $E$ is of finite measure. We set $r_x=2^nr_\varepsilon$,
where $n$ is the smallest positive integer such that
\begin{equation}\label{eqn:yla}
 \frac{\mu(E\cap
 B(x,r_x))}{\mu(B(x,r_x))}\leq\varepsilon.
\end{equation}
By~\eqref{E-finite} such an $n$ exists, and  by the choice of
$r_x$ and~\eqref{lowerbound}, it follows that
\begin{equation}\label{eqn:ala}
 \frac{\mu(E\cap
 B(x,r_x/2))}{\mu(B(x,r_x/2))}>\varepsilon.
\end{equation}
Now the doubling property of $\mu$ together with~\eqref{eqn:yla}
and~\eqref{eqn:ala} implies that
\begin{equation}\label{eqn:r_x:n_valinta}
 \varepsilon\mu(B(x,r_x))/c_D < \mu(E\cap B(x,r_x))\leq \varepsilon\mu(B(x,r_x)).
\end{equation}

 Now let $\mathcal F_1$ be a pairwise disjoint subfamily
of the balls $\{B(x,r_x)\}_{x\in E}$ given by the $5$--covering
Theorem~\ref{lemma:covering}. Note that we are actually only able
to cover Lebesgue points of $E$ but it is enough, since
$\mu$--almost every point is a Lebesgue point.

Now let $E_1$ and $\widetilde E_1$ be defined by~\eqref{E_k}.
Next, we replace $E$ by $\widetilde{E}_1$ and construct $\mathcal
{F}_2$ the same way as we constructed $\mathcal {F}_1$. Moreover,
we repeat the procedure $k_0$ times and construct $\mathcal F_k$
by replacing $E$ above by $\widetilde E_{k-1}$.

Next, we show that $\widetilde E_{k_0}\subset 5B_0$. Assume, by
contradiction, that there exists $m\leq k_0$ such that $\widetilde
E_{m}\nsubseteq 5B_0$. Then there exist balls
$B(x_k,r_k)\in\mathcal F_k$, $k=1,2,\ldots m$ such that
\[
x_{k+1}\in B(x_k,5r_k),\qquad k=1,2,\ldots, m-1,
\]
and $B(x_{m},5r_m)\nsubseteq 5B_0$. Since $E\subset B_0$, we also
know that $B(x_1,r_1)$ intersects $B_0$. This implies that
\begin{equation}\label{suuri_summa}
 \sum_{k=1}^m 5 r_k > 4R_0.
\end{equation}
Note also that since $\widetilde E_{m-1}\subset 5B_0$, also
$x_m\in 5B_0$. Next, we need an estimate for the measure of
$\widetilde E_k$. First, by~\eqref{eqn:r_x:n_valinta}, we obtain
\[
\begin{split}
\mu(\widetilde{E}_k)&= \mu\big(\bigcup_{B\in\mathcal F_k} 5B\big)
\leq c_D\sum_{B\in\mathcal F_k}\mu(B) \leq
\frac{c_D^2}{\eps}\sum_{B\in\mathcal F_k}\mu(\widetilde
E_{k-1}\cap B)\\
& =\frac{c_D^2}{\eps}\mu(\widetilde E_{k-1}\cap E_k) \leq
\frac{c_D^2}{\eps}\mu(\widetilde E_{k-1}).
\end{split}
\]
Write $\widetilde E_0 =E$. By iterating the above inequality and
by using~\eqref{eqn:k_0} we get
\begin{equation}
\label{ek-upperbound}
\begin{split}
\mu(\widetilde{E}_k)\leq \bigg(\frac{c_D^2}{\eps}\bigg)^k\mu(E)
\leq \bigg(\frac{\eps}{c_D^2}\bigg)^{k_0+1-k}\mu(B_0) 
\end{split}
\end{equation}
for $k=0,1,\ldots,k_0$.

The doubling property of $\mu$ and the fact that $x_k\in 5B_0$,
$k=1,2,\ldots m$, implies that there exists $s>0$ depending only
on the doubling constant of $\mu$ such that
\[
\bigg(\frac{5r_k}{5R}\bigg)^s\leq C
\frac{\mu(5B_k)}{\mu(5B_0)}\leq C \frac{\mu(\widetilde
E_k)}{\mu(B_0)} \leq C\bigg(\frac{\eps}{c_D^2}\bigg)^{k_0+1-k}
\leq C\varepsilon^{k_0+1-k}
\]
for $k=1,2,\ldots m$. The last inequality above follows
from~\eqref{ek-upperbound}. From this we deduce that
\[
 \sum_{k=1}^m 5r_k\leq C \sum_{k=1}^{k_0}\varepsilon^{(k_0+1-k)/s}R\leq C\varepsilon^{1/s}R.
\]
Here constant $C$ depends only on the doubling constant. If $\mu$
and $\nu$ satisfy $(1)$ for some $0<\eps<1$, they satisfy it for
all smaller $\eps$ as well. Thus we can assume that $\eps$ is
small enough to guarantee that the right--hand side of the above
inequality is less than $4R$. However, this
contradicts~\eqref{suuri_summa}. Thus $\widetilde E_k\subset 5B_0$
for all $k=1,2,\ldots, k_0$.

Next, we verify that $\{\mathcal F_k\}_{k=1}^{k_0}$ is a
telescoping sequence of collections of balls. First, by
construction, the balls in $\mathcal F_k$ are pairwise disjoint
for $k=1,2,\ldots,k_0$.

Finally, if $B(x,r)\in\mathcal F_{k-1}$, then $B(x,5r)\subset
\widetilde E_{k-1}$. Hence, in the construction of $\mathcal F_k$,
$r_x\geq 5r$, since~\eqref{eqn:yla} with $E$ replaced by
$\widetilde E_{k-1}$ cannot hold for any smaller radius.
As $B(x,r_x)$ is one of the balls that is available when we use
the covering argument to choose $\mathcal F_k$, we have
$B(x,5r)\subset B(x,r_x)\subset 5B$ for some $B\in \mathcal F_k$.
This shows that also the second condition for telescoping
sequences holds and thus the collection is telescoping.

Since $\mu$ and $\nu$ satisfy condition~\eqref{comparability}, we
conclude from~\eqref{eqn:r_x:n_valinta}  with $E$ replaced by
$\widetilde E_{k-1}$ that~\eqref{eqn:nu-estimaatti} holds for
every $k=1,2,\ldots k_0$. This completes the proof of
\eqref{comparability} $\Rightarrow$ \eqref{ainfinity}.

Since \eqref{comparability} is symmetric with respect to $\mu$ and
$\nu$, and \eqref{reversed} is \eqref{ainfinity} with the roles of
$\mu$ and $\nu$ interchanged, we also get \eqref{comparability}
$\Rightarrow$  \eqref{reversed}.
\end{proof}

If the function $r\mapsto \mu(B(x,r))$ is continuous for every $x\in X$, then it is rather easy to show that the condition~\eqref{comparability} implies that $\nu$ is doubling and consequently the conditions \eqref{comparability}--\eqref{ap} are all equivalent, see Theorem 17 on page 9 in~\cite{StroTorc89}. In particular if $X$ is a geodesic space and $\mu$ doubling, the conditions are equivalent. However, the following example shows that there are doubling metric measure spaces supporting a $(1,1)$-Poincar\'e inequality where $r\mapsto \mu(B(x,r))$ is not always continuous.
It would be interesting to know whether the conditions \eqref{comparability}--\eqref{ap} are necessarily equivalent in this type of spaces.
\begin{example}
Let $X=S\cup H\subset \R^n$, where $S$ is a unit sphere centered at the origin and $H$ is a $(n-1)$--dimensional hyperplane that contains the origin. We endow the space with the $(n-1)$--dimensional Lebesgue measure and the metric inherited from $\R^n$. Now the function $r\mapsto \mu(B(0,r))$ has a discontinuity at $r=1$. 
\end{example}

A combination of the results in this section gives us the following:
\begin{corollary}
If $\nu$ is a doubling measure, then the conditions \eqref{comparability}--\eqref{ap} are equivalent. In particular, if $\dd\nu=\omega\dd\mu$ with $\omega\in SA_\infty$, then $\mu$ and $\nu$ satisfy all the conditions.
\end{corollary}

\begin{remark}\label{remark}
If $p<\infty$ then for every  $\omega\in A_p$, the measure $\dd\nu=\omega \dd\mu$ is doubling. 
Therefore, Lemma~\ref{lemma3.3} and Theorem~\ref{theorem13} imply that an $A_\infty$--weight is an $A_p$--weight for some $p<\infty$ if and only if the weight defines a doubling measure.
\end{remark}

Now, by Remark~\ref{remark}, we obtain the following.
\begin{corollary}
 Every strong $A_\infty$--weight is an  $A_p$--weight for some $p<\infty$.
\end{corollary}

\nocite{BaloKoskRogo07}
\nocite{Maas08}
\nocite{Semm93}
\nocite{StroTorc89}
\bibliographystyle{plain}
\bibliography{Viitteeta}
\end{document}